\documentclass[12pt]{article}

\usepackage{amsmath}
\usepackage{amsfonts}
\usepackage{latexsym}
\usepackage{graphicx}
\usepackage{amssymb}
\usepackage{amsthm}
\usepackage[margin=2cm]{geometry}
\usepackage{url}
\usepackage{color}
\usepackage{enumerate}
\usepackage[shortlabels]{enumitem} %pour commencer les enumerations a des nombres differents
\usepackage[small]{caption}
\usepackage{cite}
\usepackage{framed}

\newtheorem{definition}{Definition}
\newtheorem{lemma}[definition]{Lemma}
\newtheorem{proposition}[definition]{Proposition}
\newtheorem{example}[definition]{Example}
\newtheorem{corollary}[definition]{Corollary}

\newtheorem{theorem}[definition]{Theorem}
\newtheorem{remark}[definition]{Remark}

\newcommand{\N}{\mathbb{N}}
\newcommand{\Z}{\mathbb{Z}}

\newcommand{\til}{\widetilde}
\newcommand{\Id}{{\mathrm{Id}}}
\newcommand{\Card}{{\mathrm{Card\,}}}
\newcommand{\Lst}{\textsc{Lst}}
\newcommand{\Fst}{\textsc{Fst}}
\newcommand{\LCS}{\textsc{lcs}}
\newcommand{\LCP}{\textsc{lcp}}

\newcommand{\alphabet}{\mathcal A}
\newcommand{\B}{\mathcal B}
\newcommand{\bu}{{\bf u}}

\newcommand{\bw}{{\bf w}}
\renewcommand{\L}{\mathcal L}
\newcommand{\emptyword}{\varepsilon}
\renewcommand{\P}{\mathcal P}

\begin{document}

\title{Palindromic sequences generated from marked morphisms
    \thanks{European Journal of Combinatorics (2016) 200--214.
    \tt{http://dx.doi.org/10.1016/j.ejc.2015.05.006}}
}

\author{{\sc S\'ebastien Labb\'e
\quad and \quad
Edita Pelantov\'a}\\  \\
\small LIAFA, Universit\'e Paris Diderot - Paris 7,\\ [-0.6ex]
\small Case 7014, 75205 Paris Cedex 13, France\\ [-0.6ex]
\small \tt labbe@liafa.univ-paris-diderot.fr\\ [.9ex]
\small Czech Technical University in Prague\\[-0.6ex]
\small Trojanova 13, 120 00 Praha 2, Czech Republic\\[-0.6ex]
\small \tt edita.pelantova@fjfi.cvut.cz}

\date{}

\maketitle

\begin{abstract}
Fixed points $\bu=\varphi(\bu)$ of marked and primitive morphisms
$\varphi$ over arbitrary alphabet are considered.
We show that if $\bu$ is palindromic, i.e., its language contains infinitely
many palindromes, then some power of $\varphi$ has a
conjugate in class~$\P$.
This class was introduced by Hof, Knill and Simon (1995) in order to study palindromic
morphic words.
Our definitions of marked and well-marked morphisms are more general than the
ones previously used by Frid (1999) or Tan (2007).
As any morphism with an aperiodic fixed point over a binary alphabet is
marked, our result generalizes the result of Tan.
Labb\'e (2014) demonstrated that already over a ternary alphabet the property of
morphisms to be marked is important for the validity of our theorem.
The main tool used in our proof is the description of bispecial factors in
fixed points of morphisms provided by Klouda (2012).


\end{abstract}

\section{Introduction}

In 1995,   Hof, Knill and Simon studied  the  spectral properties
of Schroedinger operators associated with one-dimensional structure
having  finite local complexity. Such structure can be coded by an
infinite word over a finite alphabet. Hof, Knill and Simon showed
that  if the  word contains infinitely  many distinct palindromes, then the
operator has  a purely singular continuous spectrum. A prominent
example of such an infinite word is the Thue-Morse word over the
binary alphabet $\{0,1\}$. This word can be obtained by the rewriting
rules:  ``$0$ replaced by  $01$" and  ``$1$ replaced by
$10$". If we start with the letter $0$ and repeat these rules,  we
get
$$ 0\mapsto 01 \mapsto 0110 \mapsto 01101001 \mapsto  01101001
10010110\mapsto  \cdots$$ The finite string we obtained in the
$i^{th}$ step is a prefix of the string we obtained in the
${(i+1)}^{st}$ step. The Thue-Morse word   is defined as the unique
infinite word  ${\bf{u}_{TM}}$  having the property that  $i^{th}$
string in our construction is a  prefix of  ${\bf{u}_{TM}}$ for each
positive integer $i$.

In Combinatorics on words, the previous construction is formalized
by using the notions {\em morphism} of  a free monoid and {\em
fixed point} of a morphism (the formal definition is provided
below). In such terminology the Thue-Morse word is a fixed point of
the morphism $\varphi_{TM}$ defined by the images of both letters in the alphabet,
namely by  $\varphi_{TM}(0)= 01$ and $\varphi_{TM}(1)= 10$. It is a
well known fact that the language of ${\bf{u}_{TM}}$, i.e., the set
of all finite strings occurring in ${\bf{u}_{TM}}$, contains
infinitely many palindromes.

In \cite{MR1372804},  a  class of  morphisms  is
 introduced: a morphism $ \varphi$ over an  alphabet $\mathcal{A}$  belongs to
 class $\P$  if  $\varphi$  has  the form
$\varphi(a)= pq_a$, where $p$ is a palindrome  and $q_a$ are
palindromes for all letters $a\in \mathcal{A}$. Hof,
Knill and Simon observed  that any fixed point of a morphism
from class $\P$  is {\em  palindromic}, i.e., infinitely many
distinct palindromes occur in the fixed point.

 In Remark 3 of
\cite{MR1372804}, the question ``Are there (minimal) sequences
containing arbitrarily long palindromes that arise from morphism
none of which  belongs to class $\P$?" is formulated.
  To guarantee that  the generated sequence is minimal (in terminology of the symbolic dynamic),  already  Hof, Knill and Simon considered primitive  morphisms only. Let us recall that a morphism $\varphi$ over an alphabet $\mathcal{A}$ is called primitive, if there exists a power $k$ such that each letter $b \in \mathcal{A}$  occurs in the  word $\varphi^k(a)$ for  any letter $a \in \mathcal{A}$. A morphism can have more fixed points. For example, $\varphi_{TM}$ has two fixed points. It  is easy to see that languages of all fixed points of a primitive morphism  coincide.

In 2003, Allouche et al. demonstrated that any  periodic palindromic
sequence is a fixed point of a morphism in class $\P$ \cite[Theorem
13]{MR1964623}.  The Thue-Morse word is palindromic but not
periodic. The morphism $\varphi_{TM}$
which generates the Thue-Morse word does not belong to  class
$\P$. Since  the Thue-Morse word is a fixed point of $\varphi_{TM}$,
it is a fixed point of the second iteration $\varphi^2_{TM}$  as
well. But the second iteration described by $\varphi_{TM}^2( 0) =
0110$ and $\varphi_{TM}^2(1)= 1001$ is in class $\P$. In
particular, the empty word serves as the palindrome $p$ and both
$q_0= 0110$, $q_1 =1001$ are palindromes as well.

The second famous example of an infinite palindromic word is  the
Fibonacci word ${\bf{u_{F}}}$. It is the fixed point of the
morphism  $\varphi_F(0)=01$ and $\varphi_F(1)=0$. Applying
$\varphi_F$ at the starting letter $0$, we get
$$0\mapsto 01\mapsto 010 \mapsto 01001 \mapsto 01001010 \mapsto  01001010 01001
\mapsto01001010 01001 01001010\cdots $$ Neither the Fibonacci morphism
$\varphi_F$
 nor its square $\varphi^2_F$ given by $\varphi^2_F(0)=010$ and
 $\varphi^2_F(1)=01$  belongs to class $\P$.
Images of both letters $0$ and $1$ under $\varphi^2_F$ start with
the same letter $0$. If we cyclically shift these images by one
letter
 to the left, i.e.,  we move the starting $0$ from the beginning
to the end,  we get the morphism given by $\psi_F(0)=100$ and
$\psi_F(1)=10$. The  morphism $\psi_F$ is already in class
$\P$.   If a  morphism $\psi$ can be obtained from a
 morphism $\varphi$ by a
 finite number of cyclical shifts, the  morphisms are
 {\em conjugate.}

  Let us stress that the property ``to be palindromic" is the property of the language of an infinite
word and not the property of the infinite word itself.   It is known (see also Lemma  \ref{lem:conjugationproperties}) that two
 primitive conjugate  morphisms have the same language.  Both  morphisms  in our  examples, namely  $\varphi_{TM}$ and  $\varphi_F$,  are obviously primitive.

 In 2007, Tan proved   that if a fixed
 point of a primitive  morphism $\varphi$  over a binary alphabet is
 palindromic, then  there exists a morphism $\psi$ in  class $\P$ such that
 languages of both  morphisms coincide \cite[Theorem 4.1]{MR2363365}.
In fact, Tan proved a stronger
 result, namely, that the  morphism  $\varphi$ or $\varphi^2$ is conjugate
 to a  morphism from class $\P$.  

In view of the previous results, in 2008, Blondin Mass\'e and Labb\'e
\cite{labbe_proprietes_2008,blondin_defaut_2008} suggested the following HKS conjecture:
 {(Version 1)\  \it  Let $\bu$ be the fixed point of a primitive morphism.
 Then,  $\bu$ is palindromic if and only if there exists a morphism
 $\varphi\neq\Id$ such that $\varphi(\bu)=\bu$  and $\varphi$
 has a conjugate in class~$\P$.}

But this statement turned out  to be false already over a ternary alphabet: in 2013, Labb\'e found an injective primitive
 morphism whose fixed points contradict Version 1 of the HKS conjecture
\cite{labbe_counterexample_2014}. In \cite{harju_remark_2013},  Harju,
Vesti and Zamboni pointed out that 
after erasing the first two letters from
 Labb\'e's
counterexample one gets   a fixed point of a  morphism from class $\P$.
  Clearly,  language of this word   coincides  with the language of the original word.
Therefore it seems that the  formulation of the theorem  which Tan stated for
a binary alphabet is more suitable  for the formulation of the HKS conjecture:
 {(Version 2) \  \it   Let $\bu$ be a fixed point of a primitive
 morphism  $\varphi$.  If $\bu$ is palindromic  then  there exists a
 morphism $\psi$ in class $\P$ such that the languages of both  morphisms coincide.  }

The possibility of different interpretations of the phrase  ``arise by
morphisms of class $\P$" in the original question of Hof, Knill and Simon, led 
the authors of \cite{harju_remark_2013} to relax the statement   ``to be
language of a fixed point of a  morphism from class $\P$"
which both previous variants impose on the language of a palindromic word
$\bu$.   Harju, Vesti and Zamboni  took into their consideration not only  a
fixed point of a primitive morphism,  but also  its   image under a further
morphism.  Such a word is  called \emph{primitive morphic}.   The last  modification of HKS conjecture sounds:      {(Version 3) \  \it  If $\bu$  is  a palindromic
primitive morphic word  then there exist morphisms $\varphi$ and  $\psi$ with conjugates in class   $\P$  and an infinite word ${\bf v}$ such that
${\bf v} = \varphi({\bf v})$ and languages of $\bu$ and $\psi({\bf v})$ coincide.}\\
Inspired
by  Labb\'e's counterexample, Harju, Vesti and Zamboni  concentrated on infinite words  with
finite defect (for definition of defect consult \cite{MR2071459}).  
  A corollary of  their  results says 
 that for any word $\bu$ with finite defect
 Version 3  is true. 
But  there exists no example of an infinite word (with finite or infinite defect) for which Version 3 holds and the stronger Version 2  fails.\\

In this paper, we  show that Version  1 of the HKS conjecture is still valid
for a large class of morphisms.  We generalize the result of Tan  for fixed
point of marked morphisms on an alphabet of arbitrary size.  A morphism $\varphi$ over an alphabet  $\mathcal{A}$ is said to be marked if it is conjugate to  morphisms  $\psi$ and $\xi$ such that the first letter of  $\psi(a)$ differs from the first letter of $\psi(b)$ and the last letter of $\xi(a)$ differs from the last letter of $\xi(b)$  for all $a,b \in \mathcal{A}$, $a\neq b$. 
Our definition of
marked morphism is more general than the ones previously used by Frid
\cite{MR1650675,MR1734902} or Tan \cite{MR2363365}.  The main result of this paper is the following.

\begin{theorem}\label{thm:main}
Let  $\bu$ be a fixed point of  a marked and primitive  morphism
$\varphi:\alphabet^*\to\alphabet^*$.  If  $\bu$ is palindromic, then  some
power of  $\varphi$ has  a conjugate in class~$\P$.
\end{theorem}
Let us stress that any  primitive morphism with an aperiodic fixed point over
a binary alphabet  is marked in our generalized sense.   Therefore,  Theorem
\ref{thm:main} generalizes Tan's result and 
it  is  another step towards a complete
characterization of the cases for which  Version  1 of the HKS conjecture
holds. 

The article is organized as follows.
In Section 2,  we recall  notions and results on morphisms, conjugacy,
cyclic and marked morphisms.  Characterization of
morphisms from class~$\P$ by the  leftmost and rightmost conjugates is provided in Section \ref{equivalentP}.
 Section \ref{bispecials} is devoted to   structure of bispecial factors of aperiodic palindromic words.
Structure of bispecial factors of marked morphisms is described in Section
\ref{wellMarked}. The crucial ingredient for our consideration in this
section is the description of bispecial factors of circular D0L morphisms given by
Klouda in \cite{MR2928192}.
Theorem~\ref{thm:main} is proved in Section \ref{proofTheorem}. The paper
ends with some comments and open questions.

\section{Preliminaries}

\subsection{Combinatorics on words}

We borrow from Lothaire \cite{MR1905123} the basic terminology about words.
In what follows, $\alphabet$ is a finite {\em alphabet} whose elements are
called  {\em letters}. A {\em word} $w$ over an alphabet $\mathcal{A}$ is a finite sequence 
$w=w_0w_1\cdots w_{n-1}$ where $n\in\N$ and  $w_i \in \mathcal{A}$ for each $i =0,1,\ldots,n-1$. 
The length of $w$ is $|w|=n$.  
%The set of all words  of length $n$ is denoted $ \alphabet^n$.  
By convention the \emph{empty word} is denoted by $\varepsilon$ and
its length is $0$. 
The set of all finite words over $\alphabet$ is denoted by $\alphabet^*$
and the set of nonempty finite words is
$\alphabet^+=\alphabet^*\setminus\{\varepsilon\}$.
Endowed with the concatenation, $\alphabet^*$ is the \emph{free monoid generated by
$\alphabet$}.
The set of right infinite words is denoted by $\alphabet^{\N}$ and the set of
bi-infinite words is $\alphabet^\Z$.
Given a word $w \in \alphabet^*\cup \alphabet^{\N}$,  a \emph{factor} $f$ of $w$ is a word
$f \in \alphabet^*$ satisfying $w = xfy$ for some $x \in \alphabet^*$ and $y
\in \alphabet^*\cup \alphabet^{\N}$.
If $x=\varepsilon$ (resp. $y=\varepsilon$) then $f$ is called a
{\em prefix} (resp. a {\em suffix}) of $w$.
If $w=vu$ then  $v^{-1}w$ and $wu^{-1}$ mean  $u$  and $v$,  respectively.
The set of all factors of $w$, called the \emph{language} of $w$, is denoted
by $\L(w)$.
% ,  and those of length  $n$ is
% $\L_n(w) = \L(w) \cap \alphabet^n.$
%Finally $\Pref(w)$ denotes the set of all prefixes of $w$.
%The number of occurrences of a factor  $f \in \alphabet^*$ is $|w|_f$.
% Given a nonempty word $w$, let $\Fst(w)$ and $\Lst(w)$ denote respectively
% the first and last letter of the word $w$.
% An infinite word  $w$ is {\em recurrent} if it satisfies the condition
% $ u \in \L(w) \implies |w|_{u}= \infty \;.$

A factor $w$ of $\bu$ is called \emph{right special}   if it has more than one
\emph{right extension} $x\in\alphabet$ such that $wx \in \L(\bu)$.
A factor $w$ of $\bu$ is called \emph{left special} if it has more than one
\emph{left extension} $x\in\alphabet$ such that $xw \in \L(\bu)$.
A factor which is both left and right special is called \emph{bispecial}.

An infinite word $\bu$ is \emph{recurrent} if any factor occurring in
$\bu$ has an infinite number of occurrences.
An infinite word $\bu$ is \emph{uniformly recurrent} if for any factor $w$ occurring
in $\bu$, there is some length $n$ such that $w$ appears in every factor of
$\bu$
of length $n$.

The {\em reversal} of $w = w_0w_1 \cdots w_{n-1} \in\alphabet^*$ is the word
$\til{\,w\,}=w_{n-1} w_{n-2}\cdots w_0$.  A {\em palindrome} is a word $w$
such that $ w=\til{\,w\,}$.
We say that the language  $\L(\bu)$ of an infinite word  $\bu$ is
\emph{closed under reversal} if  $w \in \L(\bu)$ implies $\til{\,w\,} \in
\L(\bu)$ as well.

% We recall the result of Fine and Wilf for doubly periodic
% words \cite[Theorem 8.1.4]{MR1905123}.
% \begin{lemma}[Fine and Wilf]\label{lem:finewilf}
% Let $w$ be a word having period $p$ and $q$. If
% $|w|\geq p+q-\gcd(p,q)$, then $\gcd(p,q)$ is also a period of $w$.
% \end{lemma}

%An \emph{antimorphism} is a function $\varphi : \alphabet^*\to\alphabet^*$
%such that $\varphi (uv)= \varphi(v) \varphi(u)$ for all $u,v \in \alphabet^*$.
%For example, the reversal $\til{\phantom{w}}:\alphabet^*\to\alphabet^*$ is an
%involutive antimorphism.  A word $w\in\alphabet^*$ fixed by an involutive
%antimorphism $\varphi$ is called a \emph{$\varphi$-palindrome}.

% An infinite sequence $\bu$ is \emph{palindromic}
% if its language $\L(\bu)$ contains arbitrarily long palindromes.

\subsection{Periods, conjugacy and palindromes}

This section gathers known results involving periods, conjugate words and
palindromes that are used in this article. We recall  the needed notions. 

If $w = u^k$ for some $k \in \mathbb{N}$ and $u \in \alphabet^*$, we say that $w$ is the $k$-th power of $u$. By $u^*$ we understand the  set of all powers of $u$.  
A {\em period} of a word $v$ is an integer $p$ such that $v$ is a prefix of a power $u^k$ and   $p=|u| < |v|$.  
An infinite word  $\bf w$ is called \emph{periodic} if $\bw  = u^\omega = uuu\cdots$ for some nonempty word $u$.  Length of $u$ is called a \emph{period} of $\bw$.  An infinite word $\bw$ is \emph{eventually periodic} if $\bw = v u^\omega$ for some 
finite words $v$  and $u\neq \varepsilon$;  their length $|v|$ and $|u|$ are called a  \emph{preperiod}
 and a \emph{period}, respectively.    An infinite word is \emph{aperiodic} if it is not eventually
periodic.

 First we recall the result of Fine
and Wilf for words having two periods.
%\todo{add proper ref}
%\cite[Theorem 8.1.4]{MR1905123}.%(this is the other fine and wilf)
\begin{lemma}[Fine and Wilf]\label{lem:finewilf2}
Let $u$ and $w$ be finite  words over an alphabet $\alphabet$. Suppose $u^h$
and $w^k$, for some integer $h$ and $k$, have a common prefix of length
$|u|+|w|-\gcd(|u|,|w|)$. Then there exists
$z\in \alphabet^{\ast}$ of length $\gcd(|u|,|w|)$
such that $u,w\in z^{\ast}$.
\end{lemma}

From Lothaire \cite{MR675953,MR1475463} we borrow  also another  useful result
about conjugate words, formally defined as follows.

\begin{definition}\label{def:conjugate}
Let $y\in\alphabet^*$  and $k \in \mathbb{N}$.
The \emph{$k$-th right conjugate} of $y$ is the word $x$ such that $x w
= w y$  for some word $w$ of length  $|w|=k$.
% Two words $u$ and $v$ are {\em conjugate} when there are words $x,y$ such that
% $u=xy$ and $v=yx$.
\end{definition}

\begin{lemma}{\rm \cite{MR675953,MR1475463}}\label{equationL}
If $z=xw=wy$ for some $x,y,w,z\in\alphabet^*$, then there exist unique words
$u,v\in\alphabet^*$ and a unique integer $i\geq 0$ such that
\begin{equation}\label{eq:lothaire}
x=uv, w = (uv)^i u, y=vu.
\end{equation}
where $|u|$, $0\leq|u|<|x|$, is the remainder and $i$ is the quotient of the
division of $|w|$ by $|x|$.
\end{lemma}
% The words $u$ and $v$ are uniquely determined.
% Indeed, if $|x|=|z|$ do not divide $|y|$, then
% $|u|$ is the remainder and $i$ is the quotient
% of the division of $|y|$ by $|x|$.
% If $|x|=|z|$ divides $|y|$, there is only one solution but two ways of writing
% it namely $u=r$, $v=\emptyword$ and $u=\emptyword$, $v=r$.
% Thus, to avoid duplicate solutions, we suppose that $|u|$ is the
% remainder, $0\leq|u|<|x|$, and $i$ is the quotient of the division of $|y|$ by
% $|x|$.

Now, we present an easy lemma about words that are product of two palindromes.
\begin{lemma}\label{lem:conjugatevu}
Let $u$, $v$ be two palindromes. If $k$ is an
integer such that
\[
k = |u| + \left\lceil\frac{|v|}{2}\right\rceil + \ell\cdot|vu|
\quad\quad\text{or}\quad\quad
k = \left\lceil\frac{|u|}{2}\right\rceil + \ell\cdot|vu|
\]
for some integer $\ell\geq0$,
then the $k$-th conjugate of $vu$ is of the form $\alpha\cdot p$ where
$\alpha\in\alphabet\cup\{\emptyword\}$ and $p$ is a palindrome.
Moreover,
it is a palindrome if and only if
$\alpha=\emptyword$
if and only if
the ceil function is applied on an integer, that is if $|v|$ is even in the first case
and if $|u|$ is even in the second case.
\end{lemma}

\begin{proof}
Write $u=s\alpha\til{s}$ and $v=t\beta\til{t}$ for some words $s,t\in\alphabet^*$
and letters (or empty word) $\alpha,\beta\in\alphabet\cup\{\emptyword\}$.
The conjugates
$c_1 = \beta\til{t}s\alpha\til{s}t$ and
$c_2 = \alpha\til{s}t\beta\til{t}s$
of $vu$ are of the desired form.
We have
\[
    c_1 \cdot \beta\til{t}s\alpha\til{s} = \beta\til{t}s\alpha\til{s} \cdot vu
    \quad\quad\text{and}\quad\quad
    c_2 \cdot \alpha\til{s} = \alpha\til{s} \cdot vu.
\]
Then $c_1$ is the $k_1$-th conjugate of $vu$ for
$k_1=|\beta\til{t}s\alpha\til{s}|= |u| + \lceil\frac{|v|}{2}\rceil$ and
$c_2$ is the $k_2$-th conjugate of $vu$ for
$k_2=|\alpha\til{s}|= \lceil\frac{|u|}{2}\rceil$.
Note that the $k+\ell|vu|$-th conjugate of $vu$ is equal
to the $k$-th conjugate of $vu$.
Finally, $c_1$ is a palindrome exactly when $|v|$ is even
and $c_2$ is a palindrome exactly when $|u|$ is even.
\end{proof}

We finish by stating results which appeared previously in
\cite[Lemma~1 and 2]{blondin_palindromic_2008}
and \cite[Proposition 6]{MR2830256}.
It was also published in a less general form in \cite[Lemma 5]{MR2071459}
and as a more general form (using involutive antimorphism) as Lemma 2.10 and
2.11 of \cite{labbe_proprietes_2008}. 

\begin{lemma}
{\rm\cite{MR2830256,MR2071459,labbe_proprietes_2008,blondin_palindromic_2008}}
\label{lem_equiv_i_iv}
Assume that $z=xw=wy$. Let  $u,v$ and
$i\geq 0$ be such that \eqref{eq:lothaire} holds. The following conditions are equivalent:
\begin{enumerate}[\rm (i)]
\item $x=\til{y}$\,;
\item $u$ and $v$ are palindromes, i.e., $\til{u}=u$ and $\til{v}=v$;
\item $z$ is a palindrome, i.e., $\til{z}=z$;
\item $xwy$  is a palindrome, i.e., $\til{xwy}=xwy$.
\end{enumerate}
Moreover, if one of the equivalent conditions above holds then
\begin{description}
\item[\rm \;\;(v)] $w$ is a palindrome, i.e., $\til{w}=w$.
\end{description}
\end{lemma}

\begin{lemma}
    {\rm\cite{labbe_proprietes_2008,blondin_palindromic_2008}}
    \label{lem_equiv_i_v}
Assume that $z=xw=wy$ with $|w|\geq|x|$. Then, conditions {\em (i)-(v)} in
Lemma \ref{lem_equiv_i_iv} are equivalent.
\end{lemma}

\subsection{Morphisms}

A {\em morphism} is a function $\varphi : \alphabet^* \to \alphabet^*$
compatible with concatenation, that is, such that $\varphi (uv)= \varphi(u)
\varphi (v)$ for all $u,v \in \alphabet^*$.
The \emph{identity morphism} on $\alphabet$ is denoted by $\Id_\alphabet$ or
simply $\Id$ when the context is clear.  
  A matrix $M$ with elements  $M_{ab} =$  number of occurrences of the letter $ a$  in $\varphi(b)$ for each $a,b \in \alphabet$   is 
called   \emph{ incidence matrix  of } $\varphi$.  Obviously,  $M$ is a $(d\times d)$-matrix, where  $d=\Card\alphabet$.    It is easy to see that 
a morphism $\varphi$ is \emph{primitive} if and only if some power of   its incidence matrix    has only positive elements.
% For $\alpha \in \alphabet$ we call {\em $\varphi$-block} (block for short if
% no confusion arises) a factor of the form $\varphi(\alpha)$.
% A morphism is called {\em uniform} when $|\varphi(\alpha)|=|\varphi(\beta)|$
% for all letters $\alpha,\beta\in\alphabet$.
A morphism can be naturally extended   to a map over $\alphabet^\N$ by 
$$
\varphi(u_0u_1u_2\cdots) =  \varphi(u_0) \varphi(u_1) \varphi(u_2) \cdots
$$

% Recall from Lothaire \cite{MR1905123} (Section 2.3.4)
% that $\varphi$ is \emph{right conjugate} of $\varphi'$,
% or that $\varphi'$ is \emph{left conjugate} of $\varphi$,
% noted $\varphi'\triangleright\varphi$, if
% there exists $w \in \alphabet^*$ such that
% \begin{equation}
% \varphi(x)w = w\varphi'(x), \quad \textrm{for all words } x \in \alphabet^*, \label{FirstCond}
% \end{equation}
% or equivalently that
% $\varphi(\alpha)w = w\varphi'(\alpha)$, for all letters $\alpha \in \alphabet$.
% Clearly, this relation is not symmetric so that we say that two
% morphisms $\varphi$ and $\varphi'$ are \emph{conjugate},
% %noted $\varphi\bowtie\varphi'$,
% if $\varphi'\triangleright\varphi$ or $\varphi\triangleright\varphi'$.  It is
% easy to see that conjugacy of morphisms is an equivalence relation.

A morphism is \emph{erasing} if the image of one of the letters is the
empty word.
If $\varphi$ is a nonerasing morphism, we define
$\Fst(\varphi):\alphabet\to\alphabet$ to be the
function such that $\Fst(\varphi)(a)$ is the first letter of $\varphi(a)$.
Similarly, let $\Lst(\varphi):\alphabet\to\alphabet$ be the
function such that $\Lst(\varphi)(a)$ is the last letter of $\varphi(a)$.
% Note that we have
% \[
% \Fst(\varphi\circ\mu) = \Fst(\varphi) \circ \Fst(\mu)
% \quad \quad
% \text{and}
% \quad \quad
% \Lst(\varphi\circ\mu) = \Lst(\varphi) \circ \Lst(\mu).
% \]
% When a total order on the alphabet is clear, we represent $\Fst(\varphi)$ and
% $\Lst(\varphi)$ as an ordered list.

A morphism $\varphi$ is \emph{prolongable at the  letter $a$} if $\varphi(a) = aw$,  where
$w$ is a nonempty word $w$.  If $\varphi$ is prolongable at $a$, then
\[
\bw = a w \varphi(w) \varphi(\varphi(w)) \cdots \varphi^{n}(w) \cdots
\]
 is a \emph{fixed point} of
$\varphi$, i.e., $\varphi(\bw)=\bw$. A morphism may  be prolongable at more letters, in other words a morphism may have more fixed points.  
The fixed point starting with  letter $a$ will be denoted $ \varphi^\infty(a)$.
%A fixed point  of a morphism is called a \emph{pure morphic word}. 
%A \emph{morphic word} is the image of a pure morphic word under a morphism. 
 As we have already mentioned,   all fixed points of a primitive  morphism
 $\varphi$  have the same language. Therefore, notation $\L(\varphi)$ is used instead of   $\L(\bu)$, where $\bu$ is a  fixed point of~$\varphi$.  
%$\Pal(\varphi)=\{w\in\L(\varphi) : w = \widetilde{w} \}$ denotes the set of
%palindromes in the language.
The \emph{reversal} of a morphism $\varphi$, denoted by $\til{\varphi}$,
is the morphism such that $\til{\varphi}(\alpha) = \til{\varphi(\alpha)}$ for
all $\alpha \in \alphabet$.
% The reversal $\til{\varphi}$ of a
% morphism $\varphi$ in class~$\P$ is conjugate to $\varphi$.
We now define formally class $\P$ morphisms.

\begin{definition}\label{classP}
A morphism $\varphi$ is in \emph{class $\P$} if there exists a palindrome $p$ such that  $p$ is a prefix of $\varphi(\alpha)$ and 
 $\varphi(\alpha)p$ is a palindrome for every $\alpha \in \alphabet$.  
\end{definition}
A  straightforward consequence of the above definition is that   the mapping $ \Psi: v\mapsto
 \varphi(v)p$  assigns to any palindrome $v$ a new palindrome  $\varphi(v)p$.  
If $\varphi$ is primitive then   $\Psi^n(\alpha)$  is a palindrome for any
letter $\alpha \in\mathcal{A}$ and any $n \in \mathbb{N}$  and clearly,
$\Psi^n(\alpha)$ belongs to $\L(\varphi)$. Therefore,   the language of any
fixed point of a primitive morphism in class $\P$ is palindromic.

\begin{remark} The primitivity of $\varphi \in \P$ is not necessary for palindromicity of  a fixed point of $\varphi$.
The following example was provided by Starosta
(personal communication, spring, 2014).
 Consider the
morphism $\varphi$ on a binary alphabet defined by $0\mapsto
000$ and $1\mapsto 10110100$. Clearly, the morphism belongs to 
class $\P$.  The fixed point $\varphi^\infty(1)$ is not uniformly recurrent
as it contains arbitrarily long blocks of zeros. It can be shown that
$\varphi^\infty(1)$  has defect  $0$  and thus contains infinitely many
palindromes.
\end{remark}

\subsection{Conjugacy of acyclic morphisms}

Recall from Lothaire \cite{MR1905123} (Section 2.3.4) that $\varphi$
is \emph{right conjugate} of $\psi$, or that $\psi$ is \emph{left
conjugate} of $\varphi$, noted $\psi\triangleright\varphi$, if there
exists $w \in \alphabet^*$ such that
\begin{equation}
 w\psi(x) = \varphi(x)w, \quad \textrm{for all words } x \in \alphabet^*, \label{FirstCond}
\end{equation}
or equivalently that
$ w\psi(\alpha) = \varphi(\alpha)w$, for all letters $\alpha \in \alphabet$.
% Clearly, this relation is not symmetric so that we say that two
% morphisms $\varphi$ and $\varphi'$ are \emph{conjugate}, noted
% $\varphi\bowtie\varphi'$, if $\varphi'\triangleright\varphi$ or
% $\varphi\triangleright\varphi'$.  It is easy to see that conjugacy of morphisms
% is an equivalence relation.
We say that the word $w$ is the \emph{conjugate word of the relation
$\psi\triangleright\varphi$}.

% cyclic morphism is used here
% \url{http://en.wikipedia.org/wiki/Free_monoid#Morphisms}
A morphism $\varphi:\alphabet^*\to\alphabet^*$ is \emph{cyclic} \cite{MR1475463}
if there exists a word $w\in\alphabet^*$ such that
$\varphi(\alpha)\in w^*$ for all $\alpha\in\alphabet$.
Otherwise, we say that $\varphi$ is \emph{acyclic}.
If $\varphi$ is cyclic and $|w|>1$, then the fixed point of $\varphi$ is $wwww\cdots$ and
is periodic.
Observe that the converse does not hold. For example, $a\mapsto aba, b\mapsto bab$ is
acyclic but both  its fixed points are periodic. We have the following statement.

\begin{lemma}\label{lem:conjugatetoitself}
A morphism is cyclic if and only if it is conjugate to itself with a nonempty
conjugate word.
\end{lemma}
\begin{proof}
If $\varphi:\alphabet^*\to\alphabet^*$ is cyclic,
then there exists a word $w\in\alphabet^*$ such that
for all $\alpha\in\alphabet$ there exists an integer $n$ such that
$\varphi(\alpha)=w^n$. If $w$ is empty, then $\varphi$ is conjugate to itself
with any word of $\alphabet^+$ as conjugate word. Suppose that $w$ is
not empty.
Therefore $\varphi(\alpha)w=w^n\cdot w = w\cdot w^n =
w\varphi(\alpha)$ for all $\alpha\in\alphabet$ and $\varphi$ is conjugate to
itself with a nonempty conjugate word $w$.

For the reciprocal, recall that $xy=yx$ if and only if $x$ and $y$ are powers
of the same word \cite[Prop. 1.3.2]{MR1475463}. Suppose there exists $w\in\alphabet^+$ such that
$\varphi(\alpha)w=w\varphi(\alpha)$ for all $\alpha\in\alphabet$.
Then, there exists a nonempty word $z_\alpha$
such that $\varphi(\alpha)$ and $w$ are powers of $z_\alpha$
for all $\alpha\in\alphabet$.
If $w=z_\alpha$ for all $\alpha\in\alphabet$, then $\varphi(\alpha)\in w^*$
for all $\alpha\in\alphabet$ and $\varphi$ is cyclic.
If there is only one letter $\beta\in\alphabet$ such that
$w=(z_\beta)^n$ with $n>1$ and $w=z_\alpha$ for all
$\alpha\in\alphabet\setminus\{\beta\}$, then $\varphi(\alpha)\in (z_\beta)^*$
for all $\alpha\in\alphabet$ and $\varphi$ is cyclic.
If there are more than one letter $\beta\in\alphabet$ such that
$w=(z_\beta)^{n_\beta}$ with $n_\beta>1$, then
from Lemma~\ref{lem:finewilf2} there exists a word
$z\in\alphabet^+$ such that
$z_\beta\in z^*$ for all those letters $\beta$
and
$\varphi(\alpha)\in z^*$
for all $\alpha\in\alphabet$ and $\varphi$ is cyclic.
\end{proof}
\begin{definition}
Let $\varphi$ be a morphism.
The \emph{rightmost conjugate of $\varphi$} is a morphism $\varphi_R$   such
that the following two conditions hold:
\begin{enumerate}[\rm (i)]
\item $\varphi_R$ is right  conjugate of $\varphi$;
\item if $\psi$ is right  conjugate to $\varphi_R$, then $\psi=\varphi_R$.
\end{enumerate}
 The \emph{leftmost conjugate of $\varphi$}  is  defined analogously and
denoted by $\varphi_L$.
\end{definition}

In other words,  $\varphi_R$ is the rightmost
conjugate of $\varphi$ if $\varphi_R$ is a right conjugate of $\varphi$ and
$\Lst(\varphi_R)$ is not constant. Also,  $\varphi_L$ is the leftmost conjugate of
$\varphi$ if $\varphi_L$ is a left conjugate of $\varphi$ and
$\Fst(\varphi_L)$ is not constant.
Some morphisms do not have a leftmost or rightmost conjugate. For example let
$f:a\mapsto abab,b\mapsto ab$ and $g:a\mapsto baba,b\mapsto ba$. We have that
$f\triangleright g\triangleright f$.

\begin{lemma}\label{lem:leftmostexistence}
If a morphism is acyclic, it has a leftmost and a rightmost conjugate.
\end{lemma}
\begin{proof}
    Suppose that a morphism $\varphi$ does not have a rightmost conjugate.
 Then there exists arbitrarily
    long  words $w$ satisfying 
 $\psi(\alpha)w =w\varphi(\alpha)$ for each $\alpha \in \mathcal{A}$. 
Take a word $w$ such that $|\varphi(\alpha)|$ divides the
    length of $w$ for all letters $\alpha\in\alphabet$.
According to Lemma~\ref{equationL},  $\psi(\alpha) =\varphi(\alpha)$  for each $\alpha \in \mathcal{A}$. This means that 
 $\varphi$ is conjugate to itself with a nonempty
    conjugate word. From Lemma~\ref{lem:conjugatetoitself}, $\varphi$ is
    cyclic.  
The argument for leftmost conjugates  is the same. 
\end{proof}

\begin{example}\label{longenough}
Consider the following morphisms:
\[
\begin{array}{ll}
\varphi_1 : a \mapsto babba, b \mapsto bab,&\varphi_5 : a \mapsto ababb, b \mapsto abb,\\
\varphi_2 : a \mapsto abbab, b \mapsto abb,&\varphi_6 : a \mapsto babba, b \mapsto bba,\\
\varphi_3 : a \mapsto bbaba, b \mapsto bba,&\varphi_7 : a \mapsto abbab, b \mapsto bab,\\
\varphi_4 : a \mapsto babab, b \mapsto bab.
\end{array}
\]
The morphism $\varphi_1$ is right conjugate to $\varphi_2$ with conjugate word
$b$ because
$\varphi_1(a)\cdot b = babbab = b\cdot\varphi_2(a)$ and
$\varphi_1(b)\cdot b = babb   = b\cdot\varphi_2(b)$.
% The morphism $\varphi_1$ is right conjugate to $\varphi_7$ because
% \[
% \begin{array}{c}
% \varphi_1(a)\cdot babbab = babba babbab = babbab \cdot\varphi_7(a),\\
% \varphi_1(b)\cdot babbab = bab   babbab = babbab \cdot\varphi_7(b).
% \end{array}
% \]
In general, the following relations are satisfied:
\[
\varphi_L=
\varphi_7\triangleright
\varphi_6\triangleright
\varphi_5\triangleright
\varphi_4\triangleright
\varphi_3\triangleright
\varphi_2\triangleright
\varphi_1
=\varphi_R.
\]
The morphism
$\varphi_7$ is leftmost conjugate
and
$\varphi_1$ is rightmost conjugate
with conjugate word $babbab$.% which is a palindrome.
\end{example}

Recall some simple properties of conjugate morphisms.

\begin{lemma}\label{lem:conjugationproperties}
Let a morphism $\varphi$  be right conjugate of a morphism $\psi$
and $w \in \mathcal{A}^*$ be the conjugate word of the relation
$\psi\triangleright\varphi$.
\begin{enumerate}

\item  For any $k\in \mathbb{N}$,  the  morphism $\varphi^k$ is right
conjugate of $\psi^k$.

\item $\varphi$ is injective if and only if $\psi$ is injective.

\item\label{primitivity}  $\varphi$ is primitive if and only if $\psi$ is primitive.

\item  If $\varphi$ is primitive, then $\mathcal{L}(\varphi)  =
\mathcal{L}(\psi)$.

\item  $\widetilde{\psi}$ is right conjugate of
$\widetilde{\varphi}$ and the corresponding  conjugate word is
$\widetilde{w}$.
 \end{enumerate}
\end{lemma}

\begin{proof}
1. Let us define recursively  $w_{(1)} = w$ and $w_{(k+1)}=
\varphi(w_{(k)})w = w\psi(w_{(k)})$ for any $k \in \mathbb{N}$.   We
show that $\varphi^k(u) w_{(k)} =   w_{(k)} \psi^k(u)$ for any $u
\in \mathcal{A}^*$.
We proceed by induction on $k$.

% the corresponding conjugate word of
%this relation is $w_{(k)} =  \varphi^{k-1}(w)\varphi^{k-2}(w)\cdots
%\varphi(w)w$. The definition of $w_{(k)}$ implies  that  $w_{(k+1)}=
%\varphi(w_{(k)})w = w\psi(w_{(k)})$.

Assume that $\varphi^k(u)w_{(k)} =
w_{(k)}\psi^k(u)$.  As $\varphi(v)w = w\psi(v)$ for any word $v$, we can apply
this relation to  $v= \varphi^k(u)w_{(k)} = w_{(k)}\psi^k(u)$. We get
$\varphi^{k+1}(u)\varphi(w_{(k)})w =  w \psi(w_{(k)})\psi^{k+1}(u)$, or
equivalently $\varphi^{k+1}(u)w_{(k+1)} =  w_{(k+1)} \psi^{k+1}(u)$.

%\varphi^{k-1}(w)\varphi^{k-2}(w)\cdots \varphi(w) w=
%\varphi^{k-1}(w)\varphi^{k-2}(w)\cdots \varphi(w) w = $

2. Let $u,v \in \mathcal{A}^*$. Then
$$
\psi(u) =  \psi(v) \
\Leftrightarrow  \ \psi(u)w = \psi(v)w   \  \Leftrightarrow \
w\varphi(u) = w\varphi(v) \  \Leftrightarrow  \ \varphi(u)
=\varphi(v)\,.
$$

3. Let us recall that a morphism is primitive if and only if there exists an
integer $k$ such that all elements of $k^{th}$ powers of its incidence matrix
are positive. Two mutually conjugate morphisms have the same incidence matrix.

4.  Let us fix an arbitrary $n \in \mathbb{N}$. We will show that   $\mathcal{L}_n(\varphi)   = \mathcal{L}_n(\psi)$, where   
$\mathcal{L}_n(\varphi)$ denotes  $\{w \in \mathcal{L}(\varphi) : |w|=n\}$ and    $ \mathcal{L}_n(\psi)$ is defined analogously.   According to  point
\ref{primitivity}, $\psi$ is primitive as well, and thus there exists a number $R(n)$
such that any factor of  $\mathcal{L}(\varphi)$ longer than  $R(n)$ contains
any factor of $\mathcal{L}_n(\varphi)$ and any factor of  $\mathcal{L}(\psi)$
longer than  $R(n)$ contains any factor of $\mathcal{L}_n(\psi)$. Let us find
$k$ such that for some letter $a$ both  words  $\varphi^k(a)$ and  $\psi^k(a)$
are longer than $2R(n)$.

As $\varphi^k$ and $\psi^k$ are conjugate, there exists a word, say $y$, such
that $\varphi^k(a)y=y\psi^k(a)$.  Using the fact that the equation $xy=yz$
implies $x=uv$ and $z=vu$  for some $u,v$, we can write $ \varphi^k(a) = uv$
and $\psi^k(a) =vu$, in particular $u,v \in   \mathcal{L}(\varphi)$ and $u,v
\in   \mathcal{L}(\psi)$.  Since $|\varphi^k(a)|\geq 2R(n)$, either $u$ or $v$
are longer than $R(n)$ and thus all elements from $\mathcal{L}_n(\varphi)$
occur in $u$ or $v$ and consequently in $\mathcal{L}(\psi)$ as well.

5. Applying reversal mapping to relation $\varphi(a)w
= w\psi(a)$ for any letter $a \in \mathcal{A}$, we get
$\widetilde{w} \widetilde{\varphi(a)} =
\widetilde{\psi(a)}\widetilde{w}$ as desired.
\end{proof}

\begin{lemma}\label{prolongBy_w}
Let $\varphi$ be a primitive  acyclic morphism. Denote  $\varphi_L$ and $\varphi_R$ the leftmost
 and the rightmost conjugate of $\varphi$, respectively.   Let $w \in
 \mathcal{A}^*$ be the conjugate word of the relation
 $\varphi_L\triangleright\varphi_R$.
If $u \in \mathcal{L}(\varphi)$ then $  \varphi_R(u)w =w \varphi_L(u)  \in    \mathcal{L}(\varphi)$.
\end{lemma}
\begin{proof}
Let $\bu$ be a fixed point of $\varphi$. Since $\varphi$ is primitive,  $u$  occurs in 
$\bu$ infinitely many times. Hence there exists arbitrarily long word $v$ such that $vu \in    \mathcal{L}(\varphi)$. Clearly,  $\varphi_L(vu) \in   \mathcal{L}(\varphi)$. Consequently,  $w^{-1}w \varphi_L(vu) = w^{-1}\varphi_R(vu)w = w^{-1}\varphi_R(v)\varphi_R(u)w \in   \mathcal{L}(\varphi)$.  Since the factor $v$ is long enough,    $|w^{-1}\varphi_R(v)| >0$ and thus  $\varphi_R(u)w \in   \mathcal{L}(\varphi)$.
\end{proof}

\subsection{Marked morphisms}

Frid \cite{MR1650675,MR1734902} defined a morphism $\varphi$ to be marked if
both $\Fst(\varphi)$ and $\Lst(\varphi)$ are injective.
On a binary alphabet, Tan \cite{MR2363365} defined a morphism $\varphi$ to
be marked if $\Fst(\varphi)$ is injective and well-marked if
$\Fst(\varphi)$ is the identity. It is convenient to extend the
definition of Frid to morphisms such that the cardinality of their
conjugacy class is larger than one. Also, we do not need that $\Fst(\varphi)$
be the identity in the proof of Lemma~\ref{lem:xpalindrome}.  Thus, we introduce the
following definitions that will be useful in the sequel.
\begin{definition}%[marked, well-marked]
\label{def:marked}
Let $\varphi$ be an acyclic morphism.
We say that $\varphi$ is \emph{marked} if
\begin{center}
$\Fst(\varphi_L)$ and $\Lst(\varphi_R)$ are injective
\end{center}
and that $\varphi$ is \emph{well-marked} if
\begin{center}
it is marked and if $\Fst(\varphi_L)=\Lst(\varphi_R)$
\end{center}
where $\varphi_L$ ($\varphi_R$ resp.) is the leftmost (rightmost resp.)
conjugate of~$\varphi$.
\end{definition}

\begin{example}
The morphism $\varphi_3 : a \mapsto bbaba, b \mapsto bba$ is acyclic.
It has a leftmost conjugate $\varphi_L=\varphi_7: a \mapsto abbab, b \mapsto
bab$ and a rightmost conjugate $\varphi_R=\varphi_1: a \mapsto babba, b
\mapsto bab$.
The morphism $\varphi_3$ is marked since
$\Fst(\varphi_L)$ and $\Lst(\varphi_R)$ are injective.
It is also well-marked since $\Fst(\varphi_L)=\Lst(\varphi_R)$.
\end{example}

\begin{lemma}\label{lem:powerwellmarked}
A marked morphism has a well-marked power.
\end{lemma}
\begin{proof}
Let us realize that  $\Fst(\varphi \circ\psi) = \Fst(\varphi)\circ\Fst(\psi)$ for each  morphisms
 $\varphi, \psi :\alphabet^*\to\alphabet^*$.  If   $\varphi$ is marked, then
$\Fst(\varphi_L)$ and $\Lst(\varphi_R)$ are permutations of the alphabet $\alphabet$.
Let $d=\Card\,\alphabet$
Then for
$k=d!$,  the permutations $ \bigl( \Fst(\varphi_L)\bigr)^{k}$ and $\bigl(
\Lst(\varphi_R)\bigr)^{k}$ are the identity.
Note that $\varphi_L^k$ is the leftmost conjugate and
$\varphi_R^k$ is the rightmost conjugate of
$\varphi^k$.
Since
$\Fst(\varphi_L^{k})
=\bigl( \Fst(\varphi_L)\bigr)^{k}
=\Id
=\bigl( \Lst(\varphi_R)\bigr)^{k}
=\Lst(\varphi_R^{k})$ the power    $\varphi^{k}$ is
well-marked.
\end{proof}

The power need not be $d!$. In fact, the least positive integer $a(d)$ for
which $p^{a(d)}=1$ for all permutations $p$ in $S_d$ is sufficient. This
sequence $(a(d))_d$ is well-known and indexed by A003418 in the
OEIS~\cite{OEIS}.  Of course, there might be an integer $N<a(d)$ such that
$\Fst(\varphi_L)^{N}=\Lst(\varphi_R)^{N}$. In particular, this happens for the
binary alphabet.

\begin{lemma}
A marked morphism is injective.
\end{lemma}

\begin{proof}
Let $\varphi_L$ be the leftmost conjugate of a marked morphism $\varphi$.
By definition $\Fst(\varphi_L)$ is injective so that $\varphi_L$ is injective.
From Lemma~\ref{lem:conjugationproperties}, injectivity is preserved by
conjugacy. Therefore $\varphi$ is injective.
\end{proof}

\section{Equivalent conditions for a morphism to be in class $\P$}\label{equivalentP}

Let $\varphi$ be an acyclic and primitive morphism.
Let $\varphi_R$ ($\varphi_L$ resp.) be the rightmost (leftmost resp.)
conjugate of $\varphi$ (their existence follows from the acyclic hypothesis,
see Lemma~\ref{lem:leftmostexistence}).
Let $w$ be the conjugate word
of the relation $\varphi_L\triangleright\varphi_R$. Then,
Equation \eqref{FirstCond} is satisfied:
\begin{equation*}
\varphi_R(x)w = w\varphi_L(x), \quad \textrm{for all words } x \in \alphabet^*.
\end{equation*}
Note that we have $|\varphi_R(x)|=|\varphi_L(x)|$ for all words $x\in
\alphabet^*$. Obviously, if a word $x\in \alphabet^*$ is such that $|\varphi_L(x)|\geq |w|$, then
$w$ is a suffix of $\varphi_L(x)$ and $w$ is a prefix of $\varphi_R(x)$.

The next lemma holds whatever is the size of the alphabet. An important
consequence is that HKS conjecture is satisfied for all morphisms satisfying
one of the equivalent conditions.

\begin{lemma}\label{lem:oneblargerthanw}
Let $\varphi$ be a morphism and
$\varphi_R$ ($\varphi_L$ resp.) be the rightmost (leftmost resp.) conjugate of
$\varphi$.
Let $w$ be the conjugate word such that
$w\varphi_L(u)=\varphi_R(u)w$ for all words $u\in\alphabet^*$.
Let $\B\subseteq\alphabet$ be the set of letters for which the image under
$\varphi$ is larger than the conjugate word~$w$, i.e.,
\[
\B = \{b\in\alphabet : |\varphi(b)|>|w| \}.
\]
For all $b\in\B$, let $p_b$ be the nonempty word such that $\varphi_L(b)=p_b w$.
Then, the following conditions are equivalent:
\begin{enumerate}[\rm (1)]
\item the $\lfloor\frac{|w|+1}{2}\rfloor$-th conjugate of $\varphi_L$ is in class $\P$;
\item $\varphi$ has a conjugate in class $\P$;
\item $\varphi$ and $\til{\varphi}$ are conjugates;
\item $\varphi_L=\til{\varphi_R}$;
\item $w$ is a palindrome and $p_b$ is a palindrome for all $b\in\B$.
\end{enumerate}
\end{lemma}

\begin{proof}
$(1)\implies(2)$: This is clear.

$(2)\implies(3)$: Let $\varphi':\alpha\mapsto pq_\alpha$ be the conjugate of
$\varphi$ in class $\P$. Then $\til{\varphi}$ is conjugate to
$\til{\varphi'}:\alpha\mapsto q_\alpha p$.  But $\varphi'$ and
$\til{\varphi'}$ are conjugates. We conclude by transitivity of the conjugacy
of morphisms.

$(3)\implies(4)$:
The leftmost conjugate of $\til{\varphi}$ is $\til{\varphi_R}$. If
$\varphi$ and $\til{\varphi}$ are conjugates, they must share the same
leftmost conjugate. Therefore $\varphi_L=\til{\varphi_R}$.

$(4)\implies(5)$:
Let $u$ be a word long enough so that $|\varphi_L(u)|\geq |w|$.
Then $w$ is a suffix of $\varphi_L(u)$ because
$w\varphi_L(u)=\varphi_R(u)w$.
But $\varphi_L(u) = \til{\varphi_R}(u) = \til{\varphi_R(\til{u})}$.
Therefore, $\til{w}$ is a prefix of $\varphi_R(\til{u})$.
But again $w$ is a prefix of $\varphi_R(\til{u})$.
We conclude that $w$ is a palindrome.
If $b\in\B$, then $\til{w}\til{p_b}w =
\til{\varphi_L}(b)w=\varphi_R(b)w=w\varphi_L(b)=wp_bw$
and $p_b$ is a palindrome.

$(5)\implies(1)$:
Let $k=\lfloor\frac{|w|+1}{2}\rfloor$ and $x_\alpha^{(k)}$ be $k$-th conjugate
of the word $\varphi_L(\alpha)$.

First, suppose $b\in\B$.
Let $z$ be the word such that $w=\widetilde{z}c z$ for some letter (or empty
word) $c\in\alphabet\cup\{\emptyword\}$.
\[
x_b^{(k)} = c z p_b \widetilde{z}
\]
is a palindrome if $|w|$ is even and
the product of the letter $c$ and a palindrome if $|w|$ is odd.

Now suppose $\alpha\in\alphabet\setminus\B$.
Then $|w|\geq |\varphi_L(\alpha)|$.
The equation $\varphi_R(\alpha)w = w\varphi_L(\alpha)$ is depicted in
Figure~\ref{fig:wlarger}.
\begin{figure}[h]
\begin{center}
\includegraphics{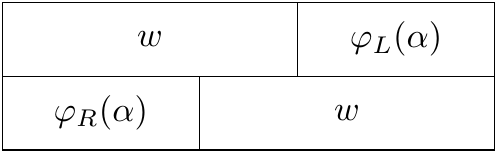}
\end{center}
\caption{When $w$ is longer than $\varphi_L(\alpha)$.}
\label{fig:wlarger}
\end{figure}
Hypothesis of Lemma~\ref{lem_equiv_i_v} are
satisfied and we conclude the existence of palindromes
$u_\alpha,v_\alpha\in\alphabet^*$ and an integer $i\geq1$ such that
$\varphi_L(\alpha)=v_\alpha u_\alpha$,
$w=(u_\alpha v_\alpha)^iu_\alpha$,
$\varphi_R(\alpha)=\widetilde{\varphi_L(\alpha)}=u_\alpha v_\alpha$ and
\[
|w| = i\cdot n_\alpha + |u_\alpha|, \quad 0\leq|u_\alpha|< n_\alpha.
\]
where $n_\alpha=|\varphi_L(\alpha)|$.
We show that the $\lfloor\frac{|w|+1}{2}\rfloor$-th conjugate of
$\varphi_L$ is in class $\P$.
Observe that $k\leq |w|$ so that the $k$-th conjugate
of $\varphi_L$ exists.
We have
\[
k=\left\lfloor\frac{|w|+1}{2}\right\rfloor
=\left\lceil\frac{|w|}{2}\right\rceil
=\left\lceil\frac{i|u_\alpha v_\alpha|+|u_\alpha|}{2}\right\rceil
=
\begin{cases}
    \displaystyle
    \left\lceil\frac{|u_\alpha|}{2}\right\rceil +\frac{i}{2}|u_\alpha
    v_\alpha| & \text{ if $i$ is even,} \\[1em]
    \displaystyle
    \left\lceil\frac{|v_\alpha|}{2}\right\rceil +|u_\alpha| +\frac{i-1}{2}|u_\alpha v_\alpha| & \text{ if $i$ is odd.}
\end{cases}
\]
If $|w|$ is even, then $i$ is odd and $|v_\alpha|$ is even or $i$ and
$|u_\alpha|$ are both even. In other words, the ceil function is applied on an integer
in the above formulas for $k$.  From Lemma~\ref{lem:conjugatevu}, the $k$-th
conjugate of the word $\varphi_L(\alpha)$ is a palindrome.
We conclude that the $k$-th conjugate of $\varphi_L$ is
mapping each letter $\alpha\mapsto x_\alpha^{(k)}$ on a palindrome, that is,
the $k$-th conjugate of $\varphi_L$ is in class $\P$.

If $|w|$ is odd, then $i$ is odd and $|v_\alpha|$ is odd or $i$ is even and
$|u_\alpha|$ is odd. In other words, the ceil function is applied on a half-integer
in the above formulas for $k$. From Lemma~\ref{lem:conjugatevu}, the $k$-th
conjugate of the word $\varphi_L(\alpha)$ is of the form $c_\alpha p$ where
$c_\alpha\in\alphabet$ is a letter and $p$ is a palindrome.
Since $k\geq 1$ and because of the definition of conjugacy, the first letter
of each word $x_\alpha^{(k)}$ is the same, i.e., $c_\alpha=c$.  We conclude
that the $k$-th conjugate of $\varphi_L$ is in class $\P$.
\end{proof}

\begin{corollary}\label{cor:wlarger}
Let  $w$ be a  conjugate word of the relation $\triangleright$ under which  a morphism  $\varphi$ is right conjugate of some  morphism.
If $w$ is a palindrome and $|w|\geq |\varphi(\alpha)|$ for all
$\alpha\in\alphabet$, then the $\lfloor\frac{|w|+1}{2}\rfloor$-th
conjugate of $\varphi$ is in class $\P$.
\end{corollary}

The morphisms  $\varphi_R$ and $\varphi_L$  from  Example \ref{longenough}
satisfy the assumptions of the previous corollary. Therefore a conjugate of
them belongs to class $\P$.

\section{Properties of palindromic words }\label{bispecials}

This section contains results on properties of palindromic words. It ends with
a lemma giving sufficient conditions for the presence of infinitely many
palindromic bispecial factors.
First we recall the relationship between two properties  of language:  ``to be
palindromic"  and ``to be closed under reversal".

\begin{lemma}\label{lem:CuR}
Let  $\bu$ be a uniformly  recurrent palindromic word.  Then its   language is closed under reversal.
\end{lemma}

\begin{proof}
  Let $v \in \mathcal{L}(\bu)$  and $|v|=n$. Since
$\bu$  is uniformly recurrent, there exists an integer $R(n) $
such that any factor $u\in \mathcal{L}(\bu)$ with $|u|\geq R(n)$
contains all factors  from $ \mathcal{L}_n(\bu)$. In particular,
any palindrome $p \in  \mathcal{L}(\bu)$ with  $|p|\geq R(n)$
contains all factors of length $n$, i.e., $v$ occurs in $p$. And
clearly $\widetilde{v}$ occurs in $p$ as well.
\end{proof}

Berstel, Boasson, Carton and Fagnot provided in \cite{berstel_infinite_2009} a
nice  example of uniformly recurrent word whose language is closed under
reversal,  but the language is not palindromic. It means that the opposite
implication in the previous lemma does not hold.  If we restrict ourselves to
eventually periodic words, the opposite implication is valid.

The next result is well-known and already present in \cite{MR1964623, MR2071459}.

\begin{lemma}\label{lem:PeriodicCuR}
Let  $\bu$ be an eventually periodic word with language closed
under reversal. Then
 $\bu$ is  periodic and there exist two palindromes $p$ and $q$  such
that  $\bu = (pq)^\omega$.
\end{lemma}

\begin{proof}

Let $\bu = z w^\omega$, where  the length of  $z$ is  minimal and  $w$  has
the shortest length among all possible periods. First we show the fact that
$w$ occurs in  $ww$ only twice:  once   as a prefix and once as a suffix.
Indeed, otherwise  there exist nonempty words $x$ and $y$ such that $w=xy=yx$.
Then  according to  Lemma \ref{equationL} the word $w$ is a  power of a
shorter  word and it  contradicts our choice of  $w$.

Since $\mathcal{L}(\bu)$ is closed under reversal, any prefix of
$\bu$ has infinitely many occurrences in $\bu$.  In particular, the prefix $zw$ occurs in the factor $w^n$ for some sufficiently large $n \in \N$. Due to the previous fact, $z$ is a suffix of $w^{n-1}$. The minimality of the length of $z$ implies that $z$ is empty.

Then $\bu$ is purely periodic, i.e.,  $\bu = w^\omega$. Closedness under
reversal gives that $\widetilde{w}$ is  a factor
of~$\bu$. Any factor of length $|w|$, in particular $\widetilde{w}$, occurs in
$ww$. Therefore there exist factors $p,q$ of $w$ such that $p\widetilde{w}
q = ww$.  Consequently $w =pq$ and    $\widetilde{w}  = qp$.
It implies that $\widetilde{p}=p$ and $\widetilde{q}=q$ and thus both $p$ and
$q$ are palindromes.
\end{proof}

\begin{lemma}\label{lem:infty_palinBranch}
Let  $\bu\in \mathcal{A}^\mathbb{N}$ be a   palindromic word. Then there  exists a
bi-infinite word ${\bf p}:=  \cdots
p_3p_2p_1p_0p_1p_2p_3\cdots $,   where $p_0 \in
\mathcal{A}\cup\{\varepsilon\} $ and $p_i\in \mathcal{A}$ for
$i\geq1$,   such that   for any nonnegative integer $m$  the string
$ p_mp_{m-1}\cdots p_0\cdots  p_{m-1}p_m$ is a palindrome occurring
in $\mathcal{L}(\bu)$.  The word  ${\bf p}$ is called {\it
infinite palindromic branch of} $\bu$.
\end{lemma}

\begin{proof}
Let us construct  a directed infinite graph $G$: vertices of $G$ are
palindromes occurring in $\bu$.  A pair of palindromes $p$ and
$q$ are connected with an  edge starting in $p$ and ending in $q$ if
there exists a letter $a$ such that $q=apa$. It is readily seen
that

-- for any palindrome $q$ with length at least 2  there exists exactly one edge  ending in
$q$;

-- for any palindrome $q$   there exists at most $\Card \mathcal{A}$ edges starting   in
$q$;

-- no edge ends in a vertex from  $\mathcal{A}\cup\{\varepsilon\}$;

-- $G$ contains no directed cycle because any edge starts in a
shorter palindrome than it ends;

--  any palindrome $p$ is reachable by an oriented path from one of vertices  from  $\mathcal{A}\cup\{\varepsilon\}$.\\
It means that $G$ is a forest with  $1+\Card \mathcal{A}$ components.
Since the language of $\bu$ contains infinitely many palindromes, at
least one of these components is an infinite tree.  According to the famous
K\H{o}nig's lemma \cite{MR1035708},   this component contains  an infinite directed
path.  Denote its starting vertex $p_0 \in
\mathcal{A}\cup\{\varepsilon\} $. The $m^{th}$-vertex $P_m$ along
the path is a palindrome  $P_m:= p_mp_{m-1}\cdots p_0\cdots
p_{m-1}p_m$.

It implies that there exists a bi-infinite word ${\bf p}:=  \cdots p_3p_2p_1p_0p_1p_2p_3\cdots $ as desired.
\end{proof}

\begin{lemma}\label{lem:infty_bisp_pal}
Let  $\bu$ be a uniformly recurrent  palindromic word. If ${\bf
u}$  is not eventually periodic,    then its  language contains
infinitely many palindromic bispecial factors.
\end{lemma}

\begin{proof}  Let $w$ be a factor of $\bu$.
Recall that $u$ is a \emph{complete return word} of $w$ in $\bu$ if $u$ is a
factor of $\bu$, $w$ is a prefix and a suffix of $u$ and  $u$ contains no
other occurrences of $w$.

First we show the existence of a bispecial word containing any factor $w$.
 Since $\bu$   is
uniformly recurrent,  the gaps between consecutive occurrences of $w$ are
bounded, or  in other words, the set of complete return words to $w$ is
finite.
Let $v$ be the longest common prefix of all complete return words to
$w$. Clearly, it has the form  $v =wV$ for some (possibly empty) factor $V$ of
$\bu$.   If $wV$ contains two occurrences of $w$,
then $wV$ is the unique complete return word to $w$ and $\bu$ is periodic
which is a contradiction. Thus $wV$ contains only one occurrence of $w$ and
$wV$ must be shorter than the shortest complete return word to $w$. Therefore,
$wV$ is right special because it is the longest common prefix of two longer
complete return words to $w$. Moreover, $wV$ is the unique right prolongation of the factor $w$ with length $|wV|$.
Analogously, the longest common suffix  of all complete return words  to $w$  has the form $Uw$, it is
left special and $Uw$ is the unique left prolongation of the factor $w$ with length $|Uw|$.
As $Uw$ is left special and $wV$ is the unique right prolongation of the factor $w$ with length $|wV|$,
the factor $UwV$ is left special as well. For the same reason,  $UwV$ is right special and thus  $UwV$ is  a bispecial factor containing the factor $w$.

Let us show that if $w$ is a palindromic factor of   $\bu$, then the bispecial
factor $UwV$ is palindromic as well. Indeed, as $\bu$ is closed under reversal
(Lemma~\ref{lem:CuR}), the set
of complete return words to the palindrome  $w$  is closed under reversal as well.
Therefore the longest common suffix of all complete return words to $w$ is
just the reversal of the longest common prefix of all complete return words to
$w$, in other words $\widetilde{U} = V$, i.e., the bispecial factor $UwV$ is a palindrome.

We have shown that for any palindrome $w$ there exists  a palindromic
bispecial factor with length at least $|w|$. Since $\bu$ contains infinitely
many palindromes,  necessarily $\bu$ contains infinitely many palindromic bispecial factors.
\end{proof}

% \Seb{Thanks, I now understand the previous proof! The proof is nice! I like
% it. The proof is long and I feel we could extract some Lemma (like saying any
% factor is contained in a bispecial factor) of it to make it shorter. But, it
% is maybe not necessary. For now, I would just leave it like that.}

\begin{remark} In the proof of the previous lemma,  we have shown that any palindrome $w$ can be extended to a bispecial  factor of the form $Uw\widetilde{U}$, where $Uw\widetilde{U}$ is  the unique palindromic extension of $w$ with length $|Uw\widetilde{U}|$.  It implies that any palindromic branch  ${\bf p}:=  \cdots p_3p_2p_1p_0p_1p_2p_3\cdots $ contains infinitely many bispecial palindromes $P_m:= p_mp_{m-1}\cdots p_0\cdots
p_{m-1}p_m$.
\end{remark}

\section{Properties of well-marked morphisms}\label{wellMarked}

\begin{lemma}\label{lem:xpalindrome}
Let $\varphi$ be  a well-marked morphism. Denote  $\varphi_L$ and $\varphi_R$ the leftmost and the rightmost conjugate of $\varphi$, respectively.
Let $w \in \mathcal{A}^*$ be the conjugate word of the relation
 $\varphi_L\triangleright\varphi_R$.
If there exist $u, v \in\alphabet^*$ such that
\begin{equation}\label{mirror}
\widetilde{\varphi_R(u)w} = \varphi_R(v)w\,,
 \end{equation} then
 $w$ is a palindrome,
  $\widetilde{u}=v$ and
$\varphi_L(a) = \widetilde{\varphi_R}(a)$ for any letter $a$ occurring in $u$.
\end{lemma}

The hypothesis of this lemma can be made more general (injectivity of
$\varphi_L$ and $\varphi_R$ instead of well-marked) but we use it only for
well-marked morphisms.

%\todo{Can we show this lemma for uniform substitutions?}

\begin{proof}
Suppose $u=u_0u_1\cdots u_n$ and $v=v_0v_1\cdots v_m$. Due to
\eqref{mirror} we   have $\widetilde{w}\widetilde{\varphi_R(u)} =
w\varphi_L(v)$. It immediately gives $\widetilde{w} = w$ and
moreover, $\widetilde{\varphi_R(u_n)}\cdots
\widetilde{\varphi_R(u_0)} = \varphi_L(v_0)\cdots  \varphi_L(v_m)$.
%
%\begin{eqnarray*}
%u &=& w \varphi_L(x)
% = w \varphi_L(x_0) \varphi_L(x_1) \cdots \varphi_L(x_n)\\
% &=& \varphi_R(x) w
% = \varphi_R(x_0) \varphi_R(x_1) \cdots \varphi_R(x_n) w\\
% &=& \widetilde{\varphi_R(x) w}
% = \widetilde{w}\widetilde{\varphi_R}(x_n) \widetilde{\varphi_R}(x_{n-1}) \cdots \widetilde{\varphi_R}(x_0)
%\end{eqnarray*}
%Clearly, $w=\widetilde{w}$ and
Hence,
\[
\Fst(\varphi_L(v_0)) = \Fst(\widetilde{\varphi_R(u_n)}) =
\Lst(\varphi_R(u_n)).
\]
Since $\varphi$ is well-marked, then
\[
\Lst(\varphi_R(u_n))=\Fst(\varphi_L(u_n)).
\]
We get $\Fst(\varphi_L)(v_0)=\Fst(\varphi_L)(u_n)$. But
$\Fst(\varphi_L)$ is injective since $\varphi$ is well-marked. We
conclude that $u_n=v_0$.  As $|\varphi_L(a)| = |\varphi_R(a)|
=|\widetilde{\varphi_R}(a)| $ for any letter $a \in \mathcal{A}$, we
also deduce    $ \varphi_L(v_0) = \widetilde{\varphi_R}(v_0)$. It
implies $ \varphi_L(v_1) \cdots \varphi_L(v_m)=
\widetilde{\varphi_R}(u_{n-1}) \cdots \widetilde{\varphi_R}(u_0) $.
For the same reason as above, we have   $v_1=u_{n-1}$ and $
\varphi_L(v_1) = \widetilde{\varphi_R}(v_1)$ and $m=n$.
\end{proof}

\begin{proposition}\label{generatingBispecial}
Let $\varphi$ be a primitive marked morphism. Denote  $\varphi_L$
and $\varphi_R$ the leftmost and the rightmost conjugate of $\varphi$,
respectively.   Let $w \in \mathcal{A}^*$ be the conjugate word of the
relation  $\varphi_L\triangleright\varphi_R$.  Define
 $\Phi:  \mathcal{L}(\varphi) \mapsto  \mathcal{L}(\varphi)$  by
$$ \Phi(u) =  \varphi_R(u)w\,. $$
\begin{enumerate}
\item If $u \in \mathcal{L}(\varphi) $ is a left (resp. right) special factor, then $ \Phi(u)$  is a left (resp. right) special factor, too.

\item  There exist a finite number of bispecial factors, say
$u^{(1)},   u^{(2)}, \ldots, u^{(N)} \in  \mathcal{L}(\varphi)$, such that
 any bispecial factor of  $ \mathcal{L}(\varphi)$ equals $\Phi^n(u^{(j)})$ for some $j =1,2,\ldots, N$ and $n \in \mathbb{N}$, where  $\Phi^n$ denotes the $n^{th}$ iteration of $\Phi$.
\end{enumerate}
\end{proposition}

The proof of the proposition is based on results of Klouda from \cite{MR2928192}.
They concern a  very broad class of circular non-pushy $D0L$-systems. Any
primitive injective morphism belongs to this class.
For two words $x,y \in \mathcal{A}^*$,  we define $\LCP\{x,y\}$
 to be the longest common prefix of $x$ and $y$   and  $\LCS\{x,y\}$   to be the longest common suffix of $x$ and $y$.
Let us summarize the relevant consequences of  Theorems 22  and 36 from
\cite{MR2928192} in the case of injective primitive morphism.

\begin{theorem}[Klouda, \cite{MR2928192}]\label{thm:klouda}  
Let $\varphi$ be an injective primitive morphism. Then,
there exist a finite set $\mathcal{I}\subset\L(\varphi)$ of bispecial factors
%(called initial factors)
and  two  finite sets  $\mathcal{B}_L$  and $ \mathcal{B}_R$
    %(called $L$-forky set  and $R$-forky set, resp.)
satisfying
\begin{align*}
    &\mathcal{B}_L \subseteq \left\{ \LCS\{ \varphi(x), \varphi(y)\} :
        x,y\in\mathcal{A}^+,
    \Lst(x)\neq\Lst(y)\right\},\\
 &\mathcal{B}_R \subseteq \left\{ \LCP\{ \varphi(x), \varphi(y)\} :
        x,y\in\mathcal{A}^+,
    \Fst(x)\neq\Fst(y)\right\}
\end{align*}
such that
any bispecial factor  $u \in \L(\varphi)\setminus\mathcal{I}$  has the form  $u=  f_L\, \varphi(u')\, f_R$, where
 $u'\in\L(\varphi) $  is a bispecial factor,  $f_L \in \mathcal{B}_L$  and $f_R \in \mathcal{B}_R$.
\end{theorem}

\begin{proof}[Proof of Proposition~\ref{generatingBispecial}]
According to Lemma~\ref{prolongBy_w}, the definition of the mapping  $\Phi$ is correct.
There exist $a, b \in \mathcal{A}$, $a\neq b$ such that  $au, bu \in \mathcal{L}(\varphi)$. According to Lemma~\ref{prolongBy_w},   words  $\varphi_R(a) \varphi_R(u)w$ and   $\varphi_R(b) \varphi_R(u)w$ belong to $\mathcal{L}(\varphi)$ too. Since $\varphi$ is marked, the last letters of  $\varphi_R(a)$ and $\varphi_R(b)$ differ and thus  $ \varphi_R(u)w$  is left special.

Analogously,  there exist $c, d \in \mathcal{A}$, $c\neq d$ such that  $uc, ud \in \mathcal{L}(\varphi)$. According to Lemma~\ref{prolongBy_w},   words  $ \varphi_R(u)\varphi_R(c)w=\varphi_R(u)w\varphi_L(c) $ and   $ \varphi_R(u)\varphi_R(d)w =\varphi_R(u)w\varphi_L(d) $ belong to $\mathcal{L}(\varphi)$ too. Since $\varphi$ is marked, the first  letters of  $\varphi_L(c)$ and $\varphi_L(d)$ differ and thus  $ \varphi_R(u)w$  is right special.

Let us apply Theorem~\ref{thm:klouda} to the morphism $\varphi_R$. Since
$\Lst(\varphi_R)$ is  injective, $\LCS\{\varphi_R(x),\varphi_R(y)\}$ is
empty for any pair of nonempty words $x,  y$  with distinct last letters.
Thus   the set $\mathcal{B}_L$ contains only the empty word.   On the
other hand,   since $\varphi_R(x)w=w\varphi_L(x)$ and
$\varphi_R(y)w=w\varphi_L(y)$  and  $\Fst(\varphi_L)$ is  injective,
we have  $\LCP\{\varphi_R(x),\varphi_R(y)\} = w$ for any 
pair of nonempty words $x,y$  with distinct first letters. Therefore,  the set $\mathcal{B}_R$ contains only the word $w$.   It means that any
bispecial factor $u$ from $\mathcal{L}(\varphi_R)\setminus\mathcal{I}$ where
$\mathcal{I}$ is finite equals to $ \varphi_R(u')w = \Phi(u')$ for some
bispecial factor $u'\in\L(\varphi_R)$.
\end{proof}

The fixed point of a cyclic primitive morphism is periodic.
As we have already illustrated on the example 
$$\xi(a)=aba \quad \text{and}\quad \xi(b)= bab$$
the converse does not hold. This  morphisms has two periodic  fixed  points
$(ab)^\omega$   and  $(ba)^\omega$. Let us point out that the shortest  period
is 2 and  both letters of the binary alphabet occur in any factor  of length 2.
Moreover,  the morphism $\xi$  is primitive and well-marked, i.e., it
satisfies the assumption of the previous propositions. The following corollary
describes morphisms of this type.

\begin{corollary}\label{formOfMorphism} Let $\varphi$ be a primitive marked
morphism over an alphabet $\mathcal{A}$  and $\bu$ be a fixed point of~$\varphi$. If $\bu$ is eventually periodic, then  there exists a word $w \in \mathcal{A}^*$ such that
$\bu = w^\omega$, $|w|=\Card\,\alphabet$, every letter of  $\mathcal{A}$
occurs exactly once in $w$ and $\varphi(w) = w^k$ for some $k\in
\mathbb{N}$.
\end{corollary}

\begin{proof}
According to Point 1 of Proposition \ref{generatingBispecial}, if
$\mathcal{L}(\varphi)$ contains  one left or right special  factor  of length
at least 1, then it contains infinitely many left and right special factors.
Since  language of an eventually periodic word  has only finitely many left
and right special factors, necessarily  there are no nonempty special factors
in $\mathcal{L}(\varphi)$ at all. Let $\bu = v w^\omega$, where $|w|$ is the
shortest period and $|v|$ is the shortest preperiod of $\bu$. Obviously,  $v$ is
empty, otherwise the first letter of $w$ is left special. For the
same reason, any letter of the alphabet occurs in $w$  exactly once.

Let us denote by $c$ the first letter of $w$.  As   $\bu$ is  a
fixed point of $\varphi$  the first letter of $\varphi(c)$ is $c$.  Since any
letter occurs in $w$ exactly once, any prefix $u$ of  $\bu $  which starts
and ends by the  letter $c$  has the form $u=w^kc$ for some $k\in \mathbb{N}$.
As  $\varphi(w)c$  is a prefix of $\varphi(ww) = \varphi(w)\varphi(w)$ which
is a prefix  $\bu $, the word
$\varphi(w)c = w^kc$  for some $k\in \mathbb{N}$.
\end{proof}

\begin{example}
The morphism $\varphi$ defined  by $a \mapsto abcabcab$, $b \mapsto cabca$ and
$c\mapsto bc$  is a primitive marked morphism over the alphabet $\{a,b,c\}$.
Its fixed point is  $(abc)^\omega$ and $\varphi(abc) = (abc)^5$.  But the
language of the fixed point  $(abc)^\omega$ unlike the fixed point
$(ab)^\omega$ of the morphisms $\xi$
%defined just before the previous Corollary,
is not palindromic.
\end{example}

\begin{corollary}\label{onlyBinaryAlphabet}
Let $\varphi$ be a primitive marked morphism over an alphabet $\mathcal{A}$
and $\bu$ be an eventually periodic fixed point of $\varphi$. If $\bu$
is palindromic, then  $\mathcal{A}$ is a binary alphabet and $\varphi$ belongs
to class $\P$.
\end{corollary}

\begin{proof}  According to Corollary \ref{formOfMorphism} $\bu = w^\omega$
and the period $w$ contains each letter of alphabet exactly once.  Due to
Lemma \ref{lem:PeriodicCuR} the period $w = pq$, where $p$ and $q$ are
palindromes. The only possibility is that $|p|=|q| =1$  and the period has the
form $w=ab$, where $a,b \in \mathcal{A}$, $a\neq b$.  Thus, the  cardinality
of the alphabet is $|w|=2$.

Corollary \ref{formOfMorphism} moreover says that $\varphi(ab) =(ab)^k$. If $\varphi(a)= (ab)^\ell$  for some $\ell <k$ then $\varphi(b)= (ab)^{k-\ell}$  and the morphism $\varphi$ is not marked. Therefore,  $\varphi(a)= (ab)^\ell a$ and
$\varphi(b)= b(ab)^{k-\ell-1}$ and obviously belongs to class $\P$.
\end{proof}

\section{Proof of Theorem~\ref{thm:main}}\label{proofTheorem}

\begin{proposition}\label{prop:leftrightmarkedhks}
Let $\varphi:\alphabet^*\to\alphabet^*$ be  a primitive and
well-marked morphism. If the language   $\L(\varphi)$ is palindromic, then $\varphi$ has a
conjugate in class $\P$.
\end{proposition}

\begin{proof}
Because of Corollary \ref{onlyBinaryAlphabet}, we can focus on $\varphi$ with aperiodic fixed points.    Since $\varphi$ is primitive,  the language   $\L(\varphi)$ is uniformly recurrent and thus there exists a constant $K$ such that any factor longer than $K$ contains all letters from  $\mathcal{A}$.
From Lemma~\ref{lem:infty_bisp_pal}, $\L(\varphi)$ contains an
infinite number of bispecial palindromes. If a bispecial palindrome
$u$  is long enough, then  according to Proposition
\ref{generatingBispecial}  there exists   $u'$ such that $u =
\varphi_R(u')w = w \varphi_L(u')$. Without loss of generality, we
can assume that $u'$ is longer than $K$ and thus all letters  of
$\mathcal{A}$ occur in $u'$.  Due to Lemma  \ref{lem:xpalindrome},
$\varphi_L = \widetilde{\varphi_R}$. This together with
Lemma~\ref{lem:oneblargerthanw} implies the statement.
\end{proof}

We may now prove the main result, namely that Version 1 of HKS Conjecture holds for
general alphabet for the case of marked morphisms.

\begin{proof}[Proof of Theorem~\ref{thm:main}]
From Lemma~\ref{lem:powerwellmarked}, $\varphi^k$ is well-marked for some
integer $k$.
From Proposition~\ref{prop:leftrightmarkedhks}, $\varphi^k$ has a conjugate in
class $\P$.
\end{proof}

The result of Tan thus becomes a corollary of this theorem.

\begin{corollary}\label{cor:binaryhks}
Let $\alphabet$ be a binary alphabet and $\varphi:\alphabet^*\to\alphabet^*$
be an acyclic and primitive morphism.
  If the language   $\L(\varphi)$ is palindromic, then $\varphi$ or
  $\varphi^2$ has a conjugate in class $\P$.
\end{corollary}

\begin{proof}
Any acyclic binary morphism $\varphi$ is marked. If $\varphi$ is not
well-marked then so is $\varphi^2$.
%The result follows from Proposition~\ref{prop:leftrightmarkedhks}.
\end{proof}

It turns out that the square $\varphi^2$ is necessary only in some particular
cases for a binary alphabet $\alphabet=\{a,b\}$, namely when
$|w|<|\varphi(a)|$ and $|w|<|\varphi(b)|$.\\

As we have already mentioned,  for a uniformly recurrent word $\bu$,  the
closedness of languages under reversal  does not imply  that  $\bu$
is palindromic.   The following proposition is  an analogy of Theorem 3.13
from \cite{MR2363365} for larger alphabet, which is stated for  binary alphabet.

\begin{proposition}\label{prop:markedMirror}
Let $\varphi:\alphabet^*\to\alphabet^*$ be  a primitive
and marked morphism. If $\L(\varphi)$  is closed under reversal,
then $\L(\varphi)$  is palindromic.
\end{proposition}

\begin{proof}
Since $\L(\varphi) = \L(\varphi^k) $,  we can  without loss of
generality assume that  $\varphi$ is well-marked. We  exploit
Proposition \ref{generatingBispecial}.  Consider a bispecial factor
$u$ which contains all letters from the alphabet and which is longer
than any initial bispecial factor on the list $u^{(1)},
\ldots,u^{(N)} $.  The same proposition says  that the factor
$\Phi(u) = \varphi_R(u)w$ is bispecial.  Since $\L(\varphi)$ is
closed under reversal, the factor $\widetilde{\varphi_R(u)w}$ is
bispecial, too. By Proposition \ref{generatingBispecial},  there
exists a bispecial factor $v$ such that $\widetilde{\varphi_R(u)w} =
\Phi(v) = \varphi_R(v)w$. Lemma \ref{lem:xpalindrome} forces
$u=\widetilde{v}$ and $\widetilde{\varphi_R}=\varphi_L$.
Thus, $\varphi$ has a conjugate in class $\P$.
\end{proof}

% \section{Proof for the binary alphabet}

% Let $\alphabet=\{a,b\}$.
% Let $\varphi:\alphabet^*\to\alphabet^*$ be an acyclic and primitive morphism.
% Let $\varphi_L$ ($\varphi_R$ resp.) be the leftmost (rightmost resp.)
% conjugate of $\varphi$.
% Let $w$ be the conjugate word
% of the relation $\varphi_L\triangleright\varphi_R$.
% Then,
% Equation \eqref{FirstCond} is satisfied:
% \begin{equation*}
% \varphi_R(x)w = w\varphi_L(x), \quad \textrm{for all words } x \in \alphabet^*.
% \end{equation*}

%\section{Conclusion}

% In the results here, the marked hypothesis
%is used three times for different reasons in the proof of
%Lemma~\ref{lem:xpalindrome}, Lemma~\ref{lem:factorization} and
%Lemma~\ref{lem:bispecial_form}. Can we ask simply $\varphi_L$ to be
%a prefix code and $\varphi_R$ to be a suffix code?

%Or can we show that if neither $\Fst(\varphi_L)$ nor $\Lst(\varphi_R)$ are
%injective, then $\varphi$ itself has a conjugate in class $\P$ if
%$|\Pal(\varphi)|=\infty$? For example, $a\mapsto aba, b\mapsto aca, c\mapsto
%bcb$.

%The morphisms $a\mapsto aab, b\mapsto cb, c\mapsto ab$ and $a\mapsto ab,
%b\mapsto ac, c\mapsto ab$ are both in class $\P$, but they are not marked. Why
%are they in class $\P$? Where do the long palindromes come from? How each of the
%Lemmas here apply or do not apply?

\section{Comments and open questions}

In our article we focused exclusively  on infinite words generated by marked morphisms. Moreover, we studied only  Version 1 of HKS conjecture. Let us  comment  some problems  we did not touch.
\begin{itemize}

\item  We believe that Version 1 of HKS conjecture can be
proved for a larger class of primitive morphisms and not only for morphisms
having a well-marked power.  For example, consider a   word  coding a
three interval exchange transformation $T$ under permutation (321). Such word  is palindromic.  If the transformation $T$ satisfies the so-called
infinite distinct orbit condition  (i.d.o.c.) introduced by Keane in \cite{MR0357739}, then a primitive morphism fixing  a coding of $T$ cannot be marked.

Is  some  power of this   morphism conjugate to
a morphism from  class $\P$?

 The counterexample to Version 1 of  HKS conjecture constructed by the
first author is an infinite word over a ternary alphabet. In fact
this word - say $\bu$ -  is the coding of a three interval
exchange transformation $T$ with permutation $(321)$.   But it does not
satisfy i.d.o.c. It means that the factor complexity
$\mathcal{C}$ of $\bu$ is bounded by $\mathcal{C}(n)\leq n+ K$
for some constant $K$. Such word is usually referred to as a
degenerate 3iet word and it is just a morphic image of a sturmian
word.

% Every  non-degenerate word coding $T$ with permutation $(321)$ has
% the factor complexity $\mathcal{C}(n)= 2n+ 1$ and its  language
% contains infinite many palindromes as well. In \cite{MR2393372} and
% \cite{MR2425618} the necessary and sufficient conditions under which
% a non-generate 3iet word is invariant under a primitive
% substitution are derived. It would be interesting either to prove
% validity of  HKS conjecture  for all non-generate 3iet words or to
% find among them a counterexample which is not morphic image of a
% binary word.

\item On  the  binary alphabet one may consider  besides reversal mapping
also  the mapping $E$ defined by $E(w_1w_2\cdots w_n) =
(1-w_n) (1-w_{n-1}) \cdots (1-w_1)$.   Obviously, the mapping $E$ is an
involutive antimorphism, i.e.,  $E^2 = {\rm Id}$ and $E(uv) = E(v)E(u)$
for all  words $u,v \in \{0,1\}^*$.

 Let
us look at  the Thue-Morse morphism  $\varphi_{TM}: \, 0\mapsto 01$, $1\mapsto 10$.  The second iteration $\varphi_{TM}^2: \, 0\mapsto
0110$, $1\mapsto 1001$ belongs to  class $\P$ and thus
the language $\mathcal{L}(\varphi_{TM})$   contains infinitely many
palindromes. On the other
hand, the language $\mathcal{L}(\varphi_{TM})$  is closed under
involutive antimorphism $E$  as well  and  it contains infinitely
many $E$-palindromes, i.e.,  factors  $w$ satisfying $E(w)=w$. In
fact, images of letters $\varphi_{TM}(0) = 01$ and
$\varphi_{TM}(1) = 10$ are both  $E$-palindromes.  The question is:
What is an $E$-analogy of class $\P$ and HKS conjecture?

\item As we have already mentioned,  Harju, Vesti and Zamboni verified  Version 3 of HKS conjecture for  words with finite defect.     Words coding symmetric exchange
transformation and also Arnoux-Rauzy words belong to most
prominent examples of  words with defect zero.
Of the same interest is the opposite  question: Which
morphisms from  class $\P$ have a fixed point with finite defect?

In this context we have to mention the conjecture stated in the last chapter
of the article \cite{MR2566171}.

{\bf Conjecture}: {\it Let $\bu$  be a fixed point of a primitive morphism
$\varphi$. If the defect of $\bu$ is finite but nonzero,  then $\bu$ is
periodic.}

In \cite{bucci_vaslet_2012}, Bucci and Vaslet provided a morphism $\varphi$
over a ternary alphabet which  contradicts this conjecture. But their morphism
$\varphi$ is not injective.       Therefore, the validity of Conjecture is
still open for morphisms over a binary alphabet or injective morphisms.
It can be deduced from
\cite{MR2566171}, that the conjecture is true if $\varphi$ is a uniform marked
morphism.

\item The main open problem remains to prove validity of Version 2 of HKS conjecture  or at least validity of its relaxed Version 3.

\end{itemize}

\section*{Acknowledgements}

We are grateful to the anonymous referees for their many valuable comments. 
The first author is supported by NSERC postdoctoral fellowship (Canada).
The second author acknowledges support of GA\v CR 13-03538S (Czech Republic).

%\section{Old Stuff}
%\input{old_stuff.tex}

%%%%%%%%%%%%%%%%
% Bibliographie %
%%%%%%%%%%%%%%%%%
\bibliographystyle{plain} 
\bibliography{biblio}

\end{document}